\documentclass[11pt,mathserif]{article}

\usepackage[utf8]{inputenc}
\usepackage{a4wide}
\usepackage[english]{babel}
\usepackage{amsfonts}
\usepackage{amsmath}
\usepackage{amssymb}
\usepackage{amsthm}
\usepackage{enumerate}
\usepackage{hyperref}
\usepackage{url}
\usepackage{comment}

\usepackage{color}

\newtheorem{theorem}{Theorem}

\newtheorem{lemma}{Lemma}
\newtheorem{corollary}{Corollary}
\theoremstyle{definition}
\newtheorem{definition}{Definition}
\newtheorem{remark}{Remark}

\newtheorem*{construction}{Construction}
\newtheorem{fact}{Fact}

\newcommand{\C}{Z}

\newcommand{\N}{{\mathbb N}}

\newcommand{\occ}[2]{|#1|_{#2}}
\newcommand{\bad}{Bad}
\everymath{\displaystyle}

\begin{document}

\title{A construction of a $\lambda$-Poisson generic sequence}

\author{Ver\'onica Becher  and Gabriel Sac Himelfarb }

\maketitle

\begin{abstract}
Years ago Zeev Rudnick defined the $\lambda$-Poisson generic sequences as the infinite sequences of symbols in a finite alphabet where the number of occurrences of long words in the initial segments follow the Poisson distribution with parameter $\lambda$. 
Although  almost all sequences, with respect to the uniform measure, are Poisson generic, no explicit instance has yet been given.
In this note we give
 a construction of an explicit $\lambda$-Poisson generic sequence
 over any alphabet and any positive $\lambda$, except for the case of the two-symbol alphabet, in which it is required that $\lambda$ be less than or equal to the  natural logarithm of~$2$.
Since $\lambda$-Poisson genericity implies Borel normality,
the constructed sequences are  Borel normal.
The same construction provides explicit instances of Borel normal sequences that are not $\lambda$-Poisson generic.
\end{abstract}

\smallskip




\section{Introduction and Statement of Results}

A real number is {\em  Poisson generic} to an integer base $b$ greater than or equal to~$2$ if  the number of occurrences of long blocks in the initial segments of its  base-$b$ expansion follow the Poisson distribution.
The definition was given years ago by   Zeev Rudnick \cite{abm,weiss2020}, who thought of it as  a property stronger than Borel normality that still holds for almost all real numbers with respect to the Lebesgue measure.\footnote{He called the notion supernormality. Personal communication from Z. Rudnick to V. Becher, 24 May 2017.} He was motivated  his result in~\cite{Rudnick} that in almost all dilates of lacunary sequences the number of elements in a random interval of the size of the mean spacing follows the Poisson law. Rudnick asked for an explicit instance of a Poisson generic real number.

Since we consider fractional expansions of real numbers in a fixed integer base, the definition of Poisson genericity  can be given for infinite sequences of symbols in a finite alphabet.
We write  $\mathbb{N}_0$ for  the set of non-negative integers, 
and $\mathbb{N}$ for the set of positive integers.
Let  $\Omega$ be an alphabet of $b$ symbols, for~$b\geq 2$.
We write  $\Omega^{\mathbb N}$  for the  set of infinite sequences of symbols in~$\Omega$.
The finite sequences of symbols in~$\Omega$ are called words and~ $\Omega^k$ denotes the set of words of length~$k$. 

We number the positions in words and  infinite sequences starting at~$1$ and we write $w[i...j]$ for the subsequence of~$w$  that begins in position~$i$ and ends in position $j$. 
For a word~$w$ we denote its length as~$|w|$. 
Given two words $w$ and $v$, the number of occurrences of $v$ in $w$ is:
    
\centerline{$
    |v|_w=\#\{1\leq i\leq |v|-|w|+1 : v[i...i+|w|-1]=w\}.
    $}
    
\noindent
For example, $|0001|_{00}=2$.

For  $x\in\Omega^{\mathbb N}$, a positive real number $\lambda$,  $i\in\mathbb N_0$ and $k\in\mathbb N$ we write $\C^\lambda_{i,k}(x)$ for the proportion of words of length $k$ that occur exactly $i$ times in $x[1..\lfloor\lambda b^k\rfloor+k-1]$,
\[
  \C^\lambda_{i,k}(x) = \frac{\#\{w \in \Omega^k:
    |x[1...\lfloor\lambda b^k\rfloor+k-1]|_{w} = i\}}{b^k}.
\]  

\begin{definition}
Let $\lambda$ be a positive real number. A sequence $x\in\Omega^{\mathbb N}$ is $\lambda$-Poisson generic if  for every  $i\in \mathbb{N}_0$,
\[
\lim_{k\rightarrow\infty}   \C^\lambda_{i,k}(x)
= e^{-\lambda} \frac{\lambda^i}{i!}.
\]
A sequence is  Poisson generic if it is  $\lambda$-Poisson generic
 for all positive real numbers~$\lambda$.
\end{definition}

The $\lambda$-Poisson generic property  can be thought of 
in terms of  random allocations of balls in  bins, where the $N=\lfloor \lambda b^k\rfloor$ initial words of length $k$ of a random sequence are the  balls, and the $b^k$ possible words of length $k$ are the  bins. These allocations are almost independent: it can be 
checked that the probability that two words in $\Omega^k$ picked uniformly at random appear in fixed overlapping positions is exactly $b^{-2k}$, as if they were independent. 
The occupancy of a random bin 
satisfies a Poisson law in the limit, the proof can be read from~\cite[Example III.10]{FS}.

Benjamin Weiss and Yuval Peres~\cite{weiss2020} proved that almost all sequences with respect to the uniform measure\footnote{The uniform measure over $\Omega^\mathbb{N}$ is the infinite product of the uniform measure over the alphabet $\Omega$. The uniform measure on $\Omega^\mathbb{N}$ coincides with the Lebesgue measure when we identify the real numbers with their fractional expansions in each given integer base.} are Poisson generic.
In fact, they proved the following stronger result: 
Consider the finite probability spaces $\Omega^k$, $k\in \mathbb{N}$, with the uniform probability measure~$\mu^k$. Fix $x\in \Omega^{\mathbb{N}}$, and define on these spaces, for each bounded Borel set $S\subset \mathbb{R}^+$, the  integer valued random variable $M_k^x(S)$ in the following way: $M_k^x(S)(\omega)$ counts 
how many times the word $\omega$ occurs in $x$  at a position in  the set $\N\cap \{b^k s: s\in S\}$.
Then, for almost all $x$ with respect to the uniform measure, $M_k^x(\cdot)$ converges in distribution to the  Poisson point process in the  positive real line as $k$ converges to infinity, see also~\cite[Theorem~1]{abm}.  Since  
\[
\C_{i,k}^{\lambda}(x)=
\mu^k\left(\{\omega\in\Omega^k: M^{x}_k((0,\lambda])(\omega)=i\}\right)
\]
it follows that  almost all $x\in \Omega^{\mathbb{N}}$ with respect to the uniform measure are Poisson generic.
Despite this result, no explicit example has yet been given.
The following is the main result of this note and its corollary gives 
 a construction of an explicit $\lambda$-Poisson generic sequence
 over any alphabet and any positive $\lambda$, except for the case of the two-symbol alphabet, in which it is required that $\lambda$ be less than or equal to the natural logarithm of~$2$.

\begin{theorem}\label{thm:main}
Let $\lambda $ be a positive real number and $\Omega$ a $b$-symbol alphabet. Let $(p_i)_{i\in \mathbb N_0}$ be a 
sequence of non-negative real numbers such that 
$\sum\limits_{i\geq 0}p_i=1$ and
$\sum\limits_{i\geq 0}ip_i=\lambda$,
and let $p_0$ be greater than or equal to $1/2$ if $b=2$.
Then, there is  a construction of an infinite sequence~$x$ over alphabet~$\Omega $, which satisfies for every $i\in\mathbb{N}_0$,
\[
\lim_{k\rightarrow\infty} 
\C^\lambda_{i,k}(x)
= p_i.
\]
\end{theorem}

By taking $p_i=e^{-\lambda} \lambda^{i}/i!$ we obtain the promised result.   In the sequel we write $\ln$ for the natural logarithm, namely, the logarithm in base $e$.

\begin{corollary}\label{cor:main}
Let  $\Omega $ be  $b$-symbol alphabet.  In case  $b=2$, fix a positive real number $\lambda$ less than  or equal to $ \ln(2)$; otherwise fix any positive real number $\lambda$.
Then, there is a construction of  a $\lambda$-Poisson generic sequence  $x\in \Omega^\mathbb N$.
\end{corollary}

To prove  Theorem~\ref{thm:main} we give a construction that consists in concatenating segments 
of any infinite de Bruijn sequence (see Definition~\ref{def:db}),
which is a sequence that satisfies that each initial segment of length 
$b^k$ is a cyclic de Bruijn word of order~$k$~\cite[Theorem~1]{BecherHeiber},. 
Our construction works by  selecting segments of this given sequence 
and  repeating them  as many times as determined 
by the probabilities~$p_i$, for every~$i\in\mathbb {N}_0$.

\begin{remark}
An infinite sequence $x=a_1a_2\ldots$ of symbols in a given alphabet  is computable  exactly when the map $k\mapsto a_k$  is computable.
Since the set of computable sequences is countable, it has uniform measure~$0$, so the existence of computable Poisson generic sequences does not necessarily follow from the fact that the set of Poisson generic sequences has full measure.
In~\cite[Theorem~2]{abm} it is shown that there exist countably many Poisson generic computable sequences.
Theorem~\ref{thm:main} yields an explicit computable instance  whenever  $(p_i)_{i\in \mathbb N}$ is a computable sequence of real numbers, which means that the map $(i,n)\mapsto $ the $n$-th digit in the base-$b$ expansion of~$p_i$, is computable.
\end{remark}

For the next result we consider  Borel's definition of  normality for sequences of symbols in a given alphabet.
An introduction to the theory of normal numbers can be read 
from~\cite{kuipers,bugeaud}.

\begin{definition}
Let $\Omega$ be a $b$-symbol alphabet, $b\geq 2$. A  sequence $x\in\Omega^{\N}$ is Borel  normal  if every word $w$ occurs  in $x$ with the same limiting frequency as every other word of the same length,
\[
\lim_{n\to \infty}
\frac{|x[1..n]|_w}{n}=b^{-|w|}.
\]
\end{definition}

In~\cite{weiss2020} Weiss  showed that $1$-Poisson genericity implies Borel normality and that the two notions do not coincide, witnessed by the famous Champernowne sequence\footnote{Bejamin Weiss first presented  this proof  at the Institute for Advanced Study, Princeton University USA on June 16, 2010, as part of  his  conference on ``Random-like behavior in deterministic systems''.
It is available at {\url{https://www.youtube.com/watch?v=8AB7591De68&ab_channel=InstituteforAdvancedStudy}}.
It was transcribed and completed in~\cite{tllucas}.}
It is immediate to see that the infinite de Bruijn sequences (see Definition~\ref{def:db}) are not $1$-Poisson generic either.
In Theorem~\ref{thm:normCriterion} we present a Borel normality criterion that generalizes this fact. 
In contrast to Theorem~\ref{thm:main}, 
this result has no limitations in the case of the two-symbol alphabet.

\begin{theorem}\label{thm:normCriterion}
     Let $\Omega$ be a $b$-symbol alphabet, $b\geq 2$, and let $x\in\Omega^{\mathbb{N}}$. We fix a positive real number $\lambda$ and define for every $i\in \mathbb{N}_0$ the numbers $p_i=\liminf_{k\rightarrow\infty}\C_{i,k}^{\lambda}(x)$. If the numbers $p_i$ satisfy  $\sum_{i\geq 0}ip_i=\lambda$ then $x$ is Borel normal to base $b$.
\end{theorem}

\begin{remark}
It is easy to verify that if numbers $p_i$ are defined as in Theorem~\ref{thm:normCriterion} 
then it is always the case that $\sum\limits_{i\geq 0} ip_i\leq \lambda$ (for example, it 
follows from Fatou's Lemma). 
\end{remark}

The following is a  consequence of Theorem~\ref{thm:normCriterion}.

\begin{corollary}\label{cor:normality}
Every  $\lambda$-Poisson generic sequence is Borel normal, but the two notions do not coincide.
The construction in Theorem~\ref{thm:main} yields infinitely many Borel normal sequences which are not $\lambda$-Poisson generic. 
\end{corollary}

\section{Proof of Theorem~\ref{thm:main}}

\subsection{The construction} \label{sec:construction}

A  cyclic de Bruijn word of order~$n$ in a $b$-symbol alphabet  is a word $w$ of length $b^n$ where each word of length $n$  occurs  exactly once in the   circular word determined by~$w$.
The classical reference is~\cite{debruijn} but they have been found independently also by I. J. Good and by N. Korobov around the same time.
Our construction is based on the following property of de Bruijn words.

\begin{lemma}[\protect{Becher and Heiber \cite[Theorem~1]{BecherHeiber}}]
\
\label{lemma:infinite}
\begin{enumerate}
    \item Every cyclic de Bruijn word of order~$n$ over an alphabet of at least three 
symbols can be extended to a cyclic de Bruijn word of order~$n+1$. 

\item Every de Bruijn word of order~$n$ in two symbols can not be extended to order~$n+1$, but it can be extended to order~$n+2$.
\end{enumerate}
\end{lemma}

For example, consider the  alphabet  $\{0,1,2\}$. 
Then, $012110022$ is a cyclic de Bruijn word of order $2$ which can be extended to
 $012110022010200011120212221$,
which is a cyclic de Bruijn word of order~$3$.

Lemma~\ref{lemma:infinite} allows us to define infinite de Bruijn sequences.

\begin{definition}\label{def:db}
An infinite de Bruijn sequence in a $b$-symbol alphabet $\Omega$, $b\geq 3$, is an infinite sequence $x\in \Omega^{\mathbb{N}}$ such that for every 
$k\in \mathbb{N}$, $x[1...b^k]$ is a cyclic de Bruijn word of order~$k$. In the case $b=2$, we say $x\in\Omega^{\mathbb{N}}$ is an infinite de Bruijn sequence if for every $k\in\mathbb{N}$, $x[1...2^{2k-1}]$ is a cyclic de Bruijn word of order~$2k-1$. 
\end{definition}

Given a real number $y\in[0,1)$, we write  $\{y\}_k$ for  the truncation to $k$ digits of the unique base-$b$ representation of $y$ which does not end in an infinite tail of $(b-1)$'s. In the sole  case $y=1$, we choose the base-$b$ representation $\sum_{i\geq 1}(b-1)b^{-i}$.

\begin{construction}
Let $\lambda $ be a positive real number and $\Omega$ a $b$-symbol alphabet , $b\geq 2$.
Let $(p_i)_{i\in \mathbb N_0}$ be a 
sequence of non-negative real numbers such that 
$\sum\limits_{i\geq 0}p_i=1$ and $\sum\limits_{i\geq 0}ip_i=\lambda$, and let $p_0\geq 1/2$ if $b=2$. 
We define $g:\mathbb{N}\rightarrow\mathbb{N}$ as 
$    g(k)= \left\lceil\frac{k}{2}\right\rceil$ and 
we define the real numbers $(p_i^k)_{i\geq 0, k\geq 1}$ inductively as follows. 
For every $i\geq 1$,
\begin{align*}
 p_i^1&=\{p_i\}_{g(1)},
\\
p_0^1 &= 1- \sum_{i\geq 1}p_i^1.
\end{align*}
And for every $k\geq 1$ and $i\geq 1$, 
\begin{align*}
 p_i^{k+1}&= \frac{1}{b}p_i^k+\left\{\frac{b-1}{b}p_i\right\}_{g(k+1)}
\\
p_0^{k+1}&=1-\sum_{i\geq 0}p_i^{k+1}.
\end{align*}

We fix  an infinite de Bruijn sequence $A$
 over the alphabet $\Omega $. We define $A_k$ to be $A[1...b^k]$. 

Given a sequence $w$ of length $b^k$
 we say $\delta$ is a block in $w$ if it is a subsequence of $w$ and $|\delta|= b^j\leq b^k$ for some $j\in\mathbb{N}_0$.
We say  that a block $\delta$ in $w$ has absolute length~$|\delta|$ and relative length  $|\delta|b^{-k}$ with respect to~$w$.

The construction works by steps. Let~$x_k$ be the output of the construction after Step~$k$. For all $k$, $x_k$ is a prefix of $x_{k+1}$. The output of the construction is the infinite word $x$ obtained as the limit of the finite words $x_k$. 
Start with $x_0$ equal to the empty word. 

\paragraph{ Step 1.}
In this first step, we consider the base-$b$ expansion of $p_i^1$, for $i\geq 1$, 
\[
p_i^1=0.c_i
\]
For each $c_i$, $i\geq 1$, we choose $c_i$ blocks of relative length $b^{-1}$ with respect to $A_{1}$, that is, blocks of length $1$ (if $c_i=0$ we don't  choose 
any blocks). The selected blocks should be non-overlapping. This is possible thanks to the fact that $\sum_{i\geq 1}p_i^1\leq 1$. We select the blocks from left to right, leaving no gaps at the beginning or in between blocks.

The output of the construction after Step 1 is the concatenation of the chosen blocks, in any order, where for every $i\geq 1$ each of the $c_i$ selected  blocks is repeated exactly $i$ times. 

\paragraph{Step k+1.}
We consider the base-$b$ expansion of $\left\{\frac{b-1}{b}p_i\right\}_{g(k+1)}$ for $i\geq 1$:
\[
\left\{\frac{b-1}{b}p_i\right\}_{g(k+1)}= 0.a_{i,1}a_{i,2}..a_{i,g(k+1)}
\]
where $a_{i,j}\in\{0,1,2,...,b-1\}$.

We select blocks in $A_{k+1}$ in the following manner:
for each $a_{i,j}$, $i\geq 1$, $j\leq g(k+1)$, we choose $a_{i,j}$ blocks of relative length $b^{-j}$ with respect to $A_{k+1}$. If $a_{i,j}=0$ we don't select any blocks. Notice that only finitely many blocks are selected. All the selected blocks should be non-overlapping. This is possible due to the fact that 
\[
\sum_{i\geq 1}
\sum_{1\leq j\leq g(k+1)}a_{i,j}\frac{1}{b^j}=\sum_{i\geq 1}\left\{\frac{b-1}{b}p_i\right\}_{g(k+1)}\leq \frac{b-1}{b}=\frac{|A_{k+1}[b^k+1...b^{k+1}]|}{b^{k+1}}\]

In the case $b=3$, we can select the blocks anywhere in $A_{k+1}$, so there could be gaps between the blocks selected at step $k$ and the ones at step $k+1$. For example, we may take blocks from $A_{k+1}[b^k+1...b^{k+1}]$. In the case $b=2$, however, we do not allow gaps.

The construction now appends the chosen blocks to $x_k$. For every $i\geq 1$, $j\leq g(k+1)$, each of the $a_{i,j}$ selected blocks is repeated exactly $i$ times. We refer to each of the chosen blocks of $A$ as constituent segments in the output $x_{k+1}$. 
We say that the concatenation of $i$-many copies of a constituent segment corresponding to $a_{i,j}$ is a run segment in the output. 
\end{construction}

\subsection{An example}
To illustrate the way the construction works, we give an example of three steps of the execution. Just to make the example more enlightening, we  set $g(k)=k$ in this section. Take $p_0=0$, $p_1=1/2$, $p_2=5/18$, $p_3= 2/9$, and $p_i=0$ for $i\geq 4$. In this case $\lambda= 31/18$. Now fix $b=3$, $\Omega=\{0,1,2\}$ and 
\[
A = 012110022010200011120212221...
\]

$p_1=0.1111...$ and $\frac{2}{3}p_1= 0.1000...$
\medskip

$p_2=0.0211...$ and $\frac{2}{3}p_2= 0.0120...$
\medskip

$p_3=0.0200... $ and $\frac{2}{3}p_3= 0.0110...$
\medskip

\noindent
\textbf{Step 1}:
\[
A=\fcolorbox[gray]{0}{0.9}{0}12\bigg|110022\bigg|010200011120212221\bigg|...
\]
\[
x_1= \fcolorbox[gray]{0}{0.9}{0}
\]
\noindent
\textbf{Step 2}:
\[
A=012\bigg|\fcolorbox[gray]{0}{0.9}{110}\fcolorbox[gray]{0}{0.7}{0}\fcolorbox[gray]{0}{0.5}{2}2\bigg|010200011120212221\bigg|...
\]
\[
x_2= 0\fcolorbox[gray]{0}{0.9}{110}\fcolorbox[gray]{0}{0.7}{0}\fcolorbox[gray]{0}{0.7}{0}\fcolorbox[gray]{0}{0.5}{2}\fcolorbox[gray]{0}{0.5}{2}\fcolorbox[gray]{0}{0.5}{2}
\]
In this case $\fcolorbox[gray]{0}{0.9}{0}$,  $\fcolorbox[gray]{0}{0.9}{110}$, $\fcolorbox[gray]{0}{0.7}{0}$ and $\fcolorbox[gray]{0}{0.5}{2}$ are the constituent segments of $x_2$.

\noindent
\textbf{Step 3}:
\[ 
A=012\bigg|110022\bigg|\fcolorbox[gray]{0}{0.9}{010200011}\fcolorbox[gray]{0}{0.7}{120}\fcolorbox[gray]{0}{0.5}{021}\fcolorbox[gray]{0}{0.7}{2}\fcolorbox[gray]{0}{0.7}{2}\fcolorbox[gray]{0}{0.5}{1}\bigg|...
\]
\[
x_3= 011000222\fcolorbox[gray]{0}{0.9}{010200011}\fcolorbox[gray]{0}{0.7}{120}\fcolorbox[gray]{0}{0.7}{120}\fcolorbox[gray]{0}{0.5}{212}\fcolorbox[gray]{0}{0.5}{212}\fcolorbox[gray]{0}{0.5}{212}\fcolorbox[gray]{0}{0.7}{2}\fcolorbox[gray]{0}{0.7}{2}\fcolorbox[gray]{0}{0.7}{2}\fcolorbox[gray]{0}{0.7}{2}\fcolorbox[gray]{0}{0.5}{1}\fcolorbox[gray]{0}{0.5}{1}\fcolorbox[gray]{0}{0.5}{1}
\]
In this case, $\fcolorbox[gray]{0}{0.7}{120}\fcolorbox[gray]{0}{0.7}{120}$ is the run segment corresponding to the constituent segment $\fcolorbox[gray]{0}{0.7}{120}$, and $\fcolorbox[gray]{0}{0.5}{212}\fcolorbox[gray]{0}{0.5}{212}\fcolorbox[gray]{0}{0.5}{212}$ is the run segment corresponding to the constituent segment $\fcolorbox[gray]{0}{0.5}{212}$.

\subsection{Correctness}
To prove the correctness of the construction we use the following  fact and five lemmas.

\begin{fact}\label{fact:(*)}
For every  $y\in[0,1)$ and every $k\geq 1$, 
$y-\frac{1}{b^k}<\{y\}_k\leq y$. In the case $y=1$, $\{y\}_k=y-\frac{1}{b^k}$.
\end{fact}

\begin{lemma}\label{lemma:base2}
Let $b=2$ and $k\geq 2$. Then, at step $k$ of the construction it is always possible to choose all necessary blocks from $A[1...2^{k-1}]$.
\end{lemma}
\begin{proof}
First of all, notice that for $k\geq 1$, the relative length with respect to $A_{k+1}$ of the blocks we need to choose at step $k+1$ is
\[
\sum_{i\geq 1}\sum_{1\leq j\leq g(k+1)}a_{i,j}\frac{1}{2^j}=\sum_{i\geq 1}\left\{\frac{1}{2}p_i\right\}_{g(k+1)}\leq \frac{1}{2}\sum_{i\geq 1} p_i = \frac{1}{2}(1-p_0)\leq \frac{1}{4} = \frac{|A[2^{k-1}+1...2^{k}]|}{2^{k+1}},
\]
where $a_{i,j}$ has the same meaning 
as in the construction, 
and we used the hypothesis $p_0\geq \frac{1}{2}$ in the last inequality.

This means that $A[2^{k-1}+1...2^k]$ has enough space to accommodate all the necessary blocks at step $k+1$. Then, we only need to check that $A[2^{k-2}+1\dots 2^{k-1}]$ is free at step $k$ for every $k\geq 2$.
We can check it inductively. In the first step, the used proportion of $A_1$ is 
\[
\sum_{i\geq 1}p_i^1\leq \sum_{i\geq 1}p_i = 1-p_0\leq \frac{1}{2},
\]
where the last inequality holds because $p_0\geq 1/2$. Then, at least half of $A_1=A[1...2]$ remains unused after step 1, so $A[2^{2-2}+1...2^{2-1}]= A[2...2]$ is free at step $2$. This proves the base case. 

Now suppose that at step $k$, $A[2^{k-2}+1...2^{k-1}]$ is free. Thanks to the first observation, we can choose all necessary blocks there. This leaves $A[2^{k-1}+1...2^k]$ free to use at step $k+1$.

\end{proof}

\begin{lemma}\label{lemma:sum}
For every $i\geq 1$ the sum of the relative lengths with respect to $A_k$ of all 
constituent segments in the output $x_k$ that are repeated exactly $i$ times is $p_i^k$. 
\end{lemma}

\begin{proof}
It can easily be checked by induction on $k$, using the definition of $p_i^k$ and the way the construction  
operates. 
If $k=1$, 
this is immediately true by Step 1 of the construction. 
Assuming the statement is true for $k$, let us see it is also 
true for $k+1$. Notice that blocks that occur in $x_k$ 
have a relative length in $A_{k+1}$ which is $1/b$ of their 
relative length in $A_k$. The extra blocks added 
contribute with $\left\{\frac{b-1}{b}p_i\right\}_{g(k+1)}$ to the sum. 
Then, the sum of the relative lengths with respect 
to $A_{k+1}$ is \[
\frac{1}{b}p_i^k+\left\{\frac{b-1}{b}p_i\right\}_{g(k+1)}= p_i^{k+1}.
\]
\end{proof}

\begin{lemma}\label{lemma:convergence}
For every $i\in \mathbb N_0$, $\lim\limits_{k\rightarrow\infty} 
p_i^k = p_i$. In fact, for every $i\geq 1$, $k\geq 1$, the following estimation holds, 
 \begin{equation*}\label{eq1}
  \tag{$\dagger$}
     p_i-\frac{k}{b^{g(k)}}\leq p_i^k\leq p_i.
 \end{equation*}
\end{lemma}
\begin{proof}
For $i\geq 1$ we prove (\ref{eq1})  by induction on $k$. If $k=1$ it follows immediately from the definition of $p_i^1$ and Fact~\ref{fact:(*)}. 
For the inductive step, notice that
\begin{align*}
p_i^{k+1}=\frac{1}{b}p_i^k+\left\{\frac{b-1}{b}p_i\right\}_{g(k+1)}\leq \frac{1}{b}p_i+\frac{b-1}{b}p_i\leq p_i.
\end{align*}
\begin{align*}
p_i-p_i^{k+1}= p_i-\left(\frac{1}{b}p_i^k+\left\{\frac{b-1}{b}p_i\right\}_{g(k+1)}\right) & \leq p_i-\frac{1}{b}\left(p_i-\frac{k}{b^{g(k)}}\right)-\left\{\frac{b-1}{b}p_i\right\}_{g(k+1)} \\
 & \leq \frac{b-1}{b}p_i-\left\{\frac{b-1}{b}p_i\right\}_{g(k+1)} + \frac{k}{b^{1+g(k)}}\\
 & \leq \frac{1}{b^{g(k+1)}}+\frac{k}{b^{1+g(k)}}\\
 & \leq \frac{k+1}{b^{g(k+1)}}.
\end{align*}

In the last inequality we used the fact that $g(k+1)\leq g(k)+1$.

In the case of $i=0$,
\begin{align*}
    |p_0^k-p_0|=\left|1-\sum_{i\geq 1}p_i^k-\left(1-\sum_{i\geq 1}p_i\right)\right|=\sum_{i\geq 1}(p_i-p_i^k).
\end{align*}
Given $\varepsilon>0$, there exists $N>0$ such that $\sum_{i\geq N+1}p_i<\frac{\varepsilon}{2}$. Then,
\[
|p_0^k-p_0|\leq \sum_{i=1}^N (p_i-p_i^k)+\sum_{i\geq N+1}p_i<\frac{kN}{b^{g(k)}}+\frac{\varepsilon}{2}.
\]
If $k$ is big enough, then $\frac{kN}{b^{g(k)}}<\frac{\varepsilon}{2}$ and
$|p_0^k-p_0|<\varepsilon$,  as desired.
\end{proof}

\begin{lemma}\label{lemma:length}
Let $x_k$ be the  word output by the construction after step $k$. Then,
\[
\lim_{k\rightarrow\infty} \frac{|\lfloor\lambda b^k\rfloor+k-1-|x_k||}{b^k}=0.
\]
\end{lemma}
\begin{proof}
By Lemma~\ref{lemma:sum}, $|x_k|= b^k \sum\limits_{i\geq 1}ip_i^k$. Then,
\[ 
\frac{|\lfloor\lambda b^k\rfloor+k-1-|x_k||}{b^k}\leq \frac{k-1}{b^k}+ \frac{|\lfloor\lambda b^k\rfloor-\lambda b^k|}{b^k}+ \left|\lambda-\sum\limits_{i\geq 1}ip_i^k\right|\leq  \frac{k}{b^k}+\left|\lambda-\sum\limits_{i\geq 1}ip_i^k\right|.
\]
It suffices to prove then that the last term converges to zero.
Recall that  $(p_i)_{i\in\mathbb N_0}$ satisfies
$\lambda = \sum\limits_{i\geq 1}ip_i$. 
Hence,
\[
\left|\lambda-\sum\limits_{i\geq 1}ip_i^k\right|=\sum\limits_{i\geq 1}i(p_i-p_i^k).
\]
Given $\varepsilon>0$ take $N$ big enough so that $\sum\limits_{i\geq N+1}ip_i<\frac{\varepsilon}{2}$.
By means of Equation~(\ref{eq1})  of
Lemma~\ref{lemma:convergence},
\[
\sum\limits_{i\geq 1}i(p_i-p_i^k)\leq \sum\limits_{i=1}^{N}i\frac{k}{b^{g(k)}}+\sum\limits_{i\geq N+1}ip_i< \frac{k}{b^{g(k)}}\frac{N(N+1)}{2}+\frac{\varepsilon}{2}.
\]
Clearly, when $k$ is large enough,  $
\frac{k}{b^{g(k)}}\frac{N(N+1)}{2}<\frac{\varepsilon}{2}$.
\end{proof}

Recall that  each of the blocks of $A$ used at step $k$  is a constituent segment in the output~$x_k$, and the concatenation of $i$-many copies of a constituent segment corresponding to $a_{i,j}$ is a run segment. 
Let $B_k$ be the number of run segments in the output $x_k$.

\begin{lemma}\label{lemma:blocks}
The quantities $B_k$ satisfy
\[
\lim_{k\rightarrow\infty}\frac{kB_k}{b^k}= 0.
\]
\end{lemma}
\begin{proof}
Notice that for every $\ell\geq 2$.
\begin{align*}
\frac{1}{b^{g(\ell)}}(B_{\ell}-B_{\ell-1})&
=\frac{1}{b^{g(\ell)}}\sum_{i\geq 1}\sum_{j=1}^{g(\ell)}a_{i,j} 
\\
&\leq \sum_{i\geq 1}\sum_{j=1}^{g(\ell)}\frac{1}{b^j}a_{i,j}
\\&
\leq \sum_{i\geq 1}\frac{b-1}{b}p_i\\
&\leq 1.
\end{align*}
Recall $g(\ell)=\left\lceil\frac{\ell}{2} \right\rceil$. Notice 
\[
B_k=B_1+\sum\limits_{\ell=2}^k B_{\ell}-B_{\ell-1}.
\]
Then we obtain 
\begin{align*}
\frac{kB_k}{b^k}\leq &\frac{ k\sum\limits_{\ell=2}^{k}b^{g(\ell)}+kB_1}{b^k}
\\&\leq \frac{ k\sum\limits_{\ell=0}^{k}b^{1+\ell/2}+kB_1}{b^k}
\\& \leq \frac{b(b^{1/2}-1)^{-1}k (b^{(k+1)/2}-1)+kB_1}{b^k},
\end{align*}
which converges to $0$ as $k$ goes to infinity. \end{proof}

\begin{remark}
Notice there are many other alternatives for the function~$g$. For instance, any function which satisfies the following conditions is a suitable choice:
\begin{itemize}
    \item $g(k)\leq k/m$ for every $k\in\mathbb{N}$, where $m$ is a constant greater than $1$.
    \item $g(k+1)\leq g(k)+1$ for every $k\in\mathbb{N}$.
    \item $\frac{k}{b^{g(k)}}\xrightarrow{k\rightarrow\infty}0$
\end{itemize}
The first and second condition ensure Lemma~\ref{lemma:blocks} and~\ref{lemma:convergence} still hold, respectively. The third condition guarantees that $\lim_{k\rightarrow\infty}p_i^k=p_i$.
For example, $g(k)=\lceil \sqrt{k}\rceil$ is a possible choice, but $g(k)=\lceil \log_b(k)\rceil$ is not because the last condition fails. 
\end{remark}

We are now ready for  the proof of Theorem~\ref{thm:main}.

\begin{proof}[Proof of Theorem~\ref{thm:main}]

First we consider the case  $i\geq 1$ and we  estimate the value of 
$ \C^\lambda_{i,k}(x)$.
Since we are only interested in the value of  $\C^\lambda_{i,k}(x)$ as $k$ tends to infinity, by Lemma~\ref{lemma:length}, it suffices to count occurrences of words in $x_k$ instead of $x[1...\lfloor\lambda b^k\rfloor]$.

Let us see that the chosen constituent segments up to step $k$ don't have any words of length $k$ in common. By Definition~\ref{def:db}, if $b\geq 3$, each word of length $k$ occurs exactly once in $A_k$  so the claim follows. If $b=2$ and $k$ is odd, then $A[1...2^k]$ is a de Bruijn word of order $k$, so it does not repeat words of length $k$. If $b=2$ and $k$ is even, $A[1...2^{k-1}]$ is a de Bruijn word of order $k-1$, so it does not repeat words of length $k-1$ (hence, it does not repeat words of length~$k$ either). Since up to step $k$ the construction has picked blocks only from $A[1...2^{k-1}]$ (thanks to Lemma~\ref{lemma:base2}), two different constituent segments share no words of length~$k$.

Therefore,  if a constituent segment $w$ of relative 
length $b^{-j}$ with respect to $A_k$ and absolute length $b^{k-j}$ is 
repeated exactly~$i$ times in~$x_k$, we may assume it contributes with 
 $b^{k-j}$
words to the numerator of the counting function $\C^\lambda_{i,k}(x)$, that is, it contributes to  
$\C^\lambda_{i,k}(x)$ with its relative length with respect to $A_k$. To see this, notice that the words of length $k$ that could make the actual count differ from this approximation belong to one of the following 
groups:

\begin{itemize}
    \item[1.] 
    For each run segment corresponding to a constituent segment of length at least $k$, we need to consider words that occur in between the constituent segments inside the run segment, which are at most $k$ different words, and words across the ending of the run segment and the beginning of the next one, which are also at most $k$.
    Then, the number of such words is at most $2k$ for each run segment.
    \item[2.] We also need to consider the number of different words in run segments composed of constituent segments of length $s<k$. In this case, the first $s$ words of length $k$ in the run segment are repeated throughout the segment. There are less than $k$ extra words that occur across the end of the run segment and the beginning of the next one . Thus, there are less than $s+k<2k$ different words for each such run segment.
\end{itemize}
The error introduced by the previous assumption is then bounded by $2kB_k$, which by Lemma~\ref{lemma:blocks} becomes negligible as $k$ tends to infinity.

By the previous approximations, we can estimate $\C^\lambda_{i,k}(x)$ as the sum of the relative lengths of all constituent segments that occur exactly $i$ times in $x_k$, which is $p_i^k$ by Lemma~\ref{lemma:sum}. 
We can then compute the limit of~$\C^\lambda_{i,k}(x)$ as $k$ tends to infinity as $\lim\limits_{k\rightarrow \infty}p_i^k$, which is equal to $p_i$, thanks to Lemma~\ref{lemma:convergence}. 

To conclude, we must consider the case $i=0$. We need to compute \[
\C^\lambda_{0,k}(x)=  \frac{\#\{w \in \Omega^k:
    |x[1...\lfloor\lambda b^k\rfloor+k-1]|_{w} = 0\}}{b^k}.
    \]
The numerator is equal to  the total number of words of length $k$ minus the number of words of length $k$ that occur at least~once. 
Using the same approximations as before, we estimate the ratio as the relative length of the unused portion of $A_k$ after step $k$, that is,
\[
  1-\sum\limits_{i\geq1}p_i^k= p_0^k.   
\]
By Lemma~\ref{lemma:convergence}, we know that
$p_0^k$ converges to $p_0$ as $k$ goes to infinity. This concludes the proof.
\end{proof}

\begin{remark}
\label{rem:lambda}
Our construction solves the problem of exhibiting a $\lambda$-Poisson generic sequence for any fixed  positive $\lambda$ for a $b$-symbol alphabet with $b\geq 3$, and for $\lambda\leq \ln(2)$ in case $b=2$. It does not, however, allow us to generate a Poisson generic sequence. This is because we use an infinite de Bruijn sequence for the construction. Suppose $b\geq 3$ and we construct~$x$ for~$\lambda = 1$. Then the  frequencies for~$\lambda = 1/b$, $i\geq 1$, satisfy,
\[ \lim_{k\rightarrow\infty}  \C^{1/b}_{i,k+1}(x)-\frac{1}{b}
 \C^1_{i,k}(x)
=
0. 
\]
 But this relation does not hold in the case of the probability mass function of the Poisson distribution:
 \[
 e^{-1/b} \frac{1}{b^i i!}\neq
 e^{-1}\frac{1}{ bi!}.
 \]
\end{remark}
In view of Remark~\ref{rem:lambda} it remains unanswered if by using a different sequence $A$ as the source for blocks, our construction could be adapted to produce a Poisson generic sequence, over any alphabet of size $b\geq 2$.

\subsection{On the case  of the two-symbol alphabet }

Infinite de Bruijn  sequences over an alphabet of more than two symbols   satisfy, for every  $k$,
that any two disjoint segments occurring before position  $b^k$
do not share words of length~$k$ nor longer. 
The two-symbol alphabet does  not guarantee this.
Consequently, Theorem~\ref{thm:main}  solves the problem of finding $\lambda$-Poisson generic sequences over the two-symbol alphabet only partially, namely for $\lambda\leq \ln(2)$.
To solve this problem completely it would suffice to use what we may call a \textit{quasi-de Bruijn sequence}, which we define as an infinite sequence $x\in\Omega^{\mathbb{N}}$ that satisfies
\[
\lim_{k\rightarrow\infty}Z_{1,k}^1(x) = 1.
\]
That is, the proportion of words of length $k$ that do not occur exactly once in the prefixes $x[1...2^k+k-1]$ converges to zero.
It is quite immediate to see that we can run the construction of Theorem~\ref{thm:main} 
but using as an input a quasi-de Bruijn sequence in the two-symbol alphabet. 

Observe that  the  infinite  de Bruijn sequences in the two-symbol alphabet  are not necessarily quasi-de Bruijn because although  their  initial segments of length  $2^{2k-1}$ are cyclic de Bruijn words, the initial segments of length  $2^{2k}$ are not.
We do not know of any  construction  in the two-symbol alphabet proved to be quasi-de Bruijn.
There is, however,  empirical evidence
supporting the hypothesis that the Eherenfeucht-Mycielski sequence~\cite{Ehrenfeucht1992} is indeed a quasi-de Bruijn sequence. 
Not much is known about this sequence: as of today, it hasn't even been proven whether the limiting frequencies of zeros and ones are equal to~$1/2$.

\section{Proof of Theorem~\ref{thm:normCriterion}}

The proof of Theorem~\ref{thm:normCriterion} is a simple generalization of  Weiss proof  that $1$-Poisson genericity implies Borel normality~\cite{weiss2020}.
It relies on a classical result due to Pyatetskii-Shapiro in 1951.

\begin{lemma}[\protect{Pyatetskii-Shapiro \cite[Theorem 4.6]{bugeaud}}]
\label{lemma:shapiro}
   Let $\Omega$ be a $b$-symbol alphabet, $b\geq 2$. Let $x\in \Omega^{\mathbb{N}}$. If there exists a positive constant $C$ such that for every $\ell\in\mathbb{N}$ and every word $w\in\Omega^\ell$,
    \[
    \limsup_{N\rightarrow\infty}\frac{|x[1...N]|_w}{N}\leq Cb^{-\ell},
    \]
    then $x$ is Borel normal. 
\end{lemma}

Let $w$ be a fixed word.  We define the set  $\bad(k, w, \varepsilon)$ as the set of
 words of length~$k$ where the frequency of~$w$ differs from the
expected frequency for Borel normality by more than $\varepsilon$.
\[
  \bad(k, w, \varepsilon) =
  \left\{v\in \Omega^k : \left|\occ{v}{w} - kb^{-|w|} \right| \geq \varepsilon k \right\}
\]
The cardinality of the set $\bad(k, w, \varepsilon)$ has exponential decay in $k$.  This follows from Bernstein's inequality and was proved  in the early  works on Borel normal numbers, such as~\cite{CopelandErdos}
and since then each work computes a similar upper bound.

\begin{lemma}
\label{lemma:HW}
Assume a $b$-symbol alphabet.
Let $k$ and $\ell$ be positive integers and let $\varepsilon$ be such that $6/\lfloor k/\ell\rfloor\leq \varepsilon\leq 1/b^\ell$. Then, for every word $w$ of length $\ell$, 
\[
|Bad(k,w,\varepsilon)| < 
4\ell b^{k+\ell}
e^{-b^{\ell}\varepsilon^2 k/(6\ell)}.
\]
\end{lemma}

We can now give the awaiting proof.

\begin{proof}[Proof of Theorem~\ref{thm:normCriterion}]
We need to prove that $x$ is Borel normal for the  $b$-symbol alphabet, given that for all $i\in \mathbb{N}_0$,
\[
  \liminf_{k\to\infty} \C^\lambda_{i,k}(x) = p_i,
\] 
\[
\sum_{i\geq 0}ip_i=\lambda.
\]

Fix a positive real number $\varepsilon$.  By hypothesis we know that $\sum_{i \geq 0} ip_i  = \lambda$.
Let $i_0$ be such that
$
  \sum_{i > i_0} ip_i < \frac{\lambda\varepsilon}{2}.
$
It follows that
\[
  \sum_{i=0}^{i_0} ip_i >
  \lambda\left(1 - \frac{\varepsilon}{2}\right).
\]

Let $k_0$ be such that for all $k > k_0$ and $0 \leq i \leq
i_0 $,
\[
  \C^\lambda_{i,k}(x) > p_i - \frac{\lambda\varepsilon}{2i_0^2}.
\]
Consider the positions from $1$ to $\lfloor \lambda b^k\rfloor $ in~$x$.  
We say a position is blamed if the word of length~$k$ starting
at that position occurs more than $i_0$ times in the prefix
$x[1...\lfloor\lambda b^k\rfloor+k-1]$ of~$x$.  For $k>k_0$, we can bound the number of blamed positions
between $1$ and $\lfloor \lambda b^k\rfloor$ in the following way:

\begin{align*}
   \lfloor \lambda b^k\rfloor -  \sum_{i=0}^{i_0} i\ \C^\lambda_{i,k}(x) b^k
     & < \lambda b^k - b^k\sum_{i=0}^{i_0} \left(ip_i - \frac{\lambda\varepsilon}{2i_0}\right) \\
    & < \lambda b^k + \frac{b^k\lambda\varepsilon}{2}-  b^k \sum_{i=0}^{i_0} ip_i \\
    & < \lambda b^k + \frac{b^k\lambda\varepsilon }{2}- b^k \lambda\left(1-\frac{\varepsilon}{2}\right)  \\
    & < \lambda\varepsilon b^k.
\end{align*}
We cover the positions from $1 $ to $\lfloor\lambda b^k\rfloor+k-1$ with
 non-overlapping words of length~$k$ such that no word starts
at a blamed position and every position that is not blamed is covered by exactly one word. We refer to these words as covering words. Notice that a covering word may contain blamed positions as long as they are not the first one.  

Occurrences of~$w$ in  $x[1...\lfloor \lambda b^k\rfloor+k-1]$ fall into one of the following categories:

\begin{itemize}\itemsep0cm
\item occurrences of $w$ starting at a blamed position.  The number of such
  occurrences is bounded by the number of blamed positions, which is at most~$\lambda\varepsilon b^k$.
\item occurrences of $w$  not fully contained in a covering word.  Since there are at most $k^{-1}\lambda b^k$ covering words, the number of these occurrences is bounded by
  $|w|k^{-1}\lambda b^k$.
\item occurrences of $w$ contained in a covering word which is in $\bad(k,w,\varepsilon)$. Each covering word can occur at most $i_0$ times, and can contain at most $k$ occurrences of~$w$. Then there are at most $i_0 k |\bad(k, w,\varepsilon)|$ occurrences of $w$ in this case. Notice that for sufficiently large $k$, $\varepsilon\geq 6/\lfloor k/|w|\rfloor$, so we can use the bound in Lemma~\ref{lemma:HW}.
\item occurrences contained in a covering word which is not in $\bad(k,w,\varepsilon)$. Each such word contains at most  $k b^{-|w|} + \varepsilon k$ occurrences of~$w$.  Since there are at most $k^{-1}\lambda b^k$ covering words, the total number of such occurrences is at most
  $\lambda b^k(b^{-|w|}+\varepsilon)$.
\end{itemize}
Combining all the upper bounds for each category yields the following
upper bound for the number of occurrences of~$w$ in  $x[1..\lfloor \lambda b^k\rfloor+k-1]$,
\[
  \frac{\occ{x[1...\lfloor\lambda b^k\rfloor+k-1]}{w}}{\lambda b^k} \leq \varepsilon + \frac{1}{k}|w|
  + \frac{4i_0 |w|b^{|w|}}{\lambda}ke^{-\varepsilon^2 k b^{|w|}/(6|w|)} + b^{-|w|} + \varepsilon.
\]
Taking limit superior we obtain
\begin{align*}
  \limsup_{k\to\infty}\frac{\occ{x[1...\lfloor\lambda b^k\rfloor+k-1]}{w}}{\lambda b^k} \leq 2\varepsilon+b^{-|w|}.
\end{align*}
Since this holds for every $\varepsilon\leq 1/b^{|w|}$, it follows that
\begin{align*}
  \limsup_{k\to\infty}\frac{\occ{x[1...\lfloor\lambda b^k\rfloor+k-1]}{w}}{\lambda b^k} \leq b^{-|w|}.
\end{align*}

To show that $x$ is Borel normal we apply Lemma~\ref{lemma:shapiro}.
Fix $N$ and let $k$ be such that
$\lambda b^{k-1} \leq N <\lambda b^k$.  Then, using the bounds obtained before,
\[
  \limsup_{N\to \infty}  \frac{\occ{x[1... N]}{w}}{N} 
    \leq \limsup_{n\to \infty}\frac{\occ{x[1...\lfloor\lambda b^{k}\rfloor+k-1]}{w}}{\lambda b^{k-1}}
    \leq b^{1-|w|}.
\]
We conclude that $x$ is Borel normal.
\end{proof}
\bigskip
\bigskip

\noindent
{\bf Acknowledgements.}
The authors are grateful to Zeev Rudnick and to  Benjamin Weiss for their comments on our investigations and  for  having insisted  that we solve this problem.
Gabriel Sac Himelfarb is supported by the student fellowship 
``Beca de Estímulo a las Vocaciones Científicas'' convocatoria 2020,   Consejo Interuniversitario Nacional, Argentina. 
Verónica Becher is supported by Agencia Nacional de Promoci\'on Cient\'ifica y Tecnol\'ogica grant PICT-2018-02315  and by\linebreak Universidad de Buenos Aires grant Ubacyt 20020170100309BA.

\bibliographystyle{plain}
\bibliography{bibliography}

\noindent

\noindent
Ver\'onica Becher \\
 Departamento de  Computaci\'on, Facultad de Ciencias Exactas y Naturales \& ICC,
 Universidad de Buenos Aires \&  CONICET,  Argentina
 \\{\tt  vbecher@dc.uba.ar}
\medskip

\noindent
Gabriel Sac Himelfarb \\
 Departamento de Matem\'atica \& Departamento de Computaci\'on, Facultad de Ciencias Exactas y Naturales,
 Universidad de Buenos Aires, Argentina\\ {\tt gabrielsachimelfarb@gmail.com}
\end{document}